\documentclass[11pt, reqno]{amsart}
\usepackage{amsmath, amssymb, amsfonts, amstext, verbatim, amsthm, mathtools}

\input xy
\xyoption{all}
\usepackage{setspace}
\usepackage{enumerate}
\usepackage{prettyref}
\usepackage{fullpage}
\usepackage{tikz-cd}
\usepackage{stmaryrd}

\newrefformat{prop}{Proposition~\ref{#1}}
\newrefformat{cor}{Corollary~\ref{#1}}
\newrefformat{thm}{Theorem~\ref{#1}}
\newrefformat{lem}{Lemma~\ref{#1}}
\newrefformat{def}{Definition~\ref{#1}}

\theoremstyle{plain}
\newtheorem{thm}{Theorem}[section]
\newtheorem*{thm*}{Theorem}
\newtheorem{lem}[thm]{Lemma}

\newtheorem{prop}[thm]{Proposition}

\theoremstyle{definition}
\newtheorem{rem}[thm]{Remark}

\newtheorem{ex}[thm]{Example}

\newcommand{\landw}{\Lambda}
\newcommand{\frakm}{\mathfrak{m}}
\newcommand{\En}{\widehat{E}(n)}

\DeclareMathOperator*{\gplus}{+}
\DeclareMathOperator*{\gdot}{\cdot}
\newcommand{\xrightarrowdbl}[2][]
{\xrightarrow[#1]{#2}\mathrel{\mkern-14mu}\rightarrow}

\title{Integral Kirwan surjectivity}
\author{Daniel Pomerleano and Constantin Teleman}

\begin{document}
\begin{abstract} We refine Kirwan's \emph{surjectivity} and \emph{formality} theorems for a Hamiltonian 
$G$-action on a compact symplectic manifold $M$. For a regular value of the 
moment map, we show that the Kirwan map is surjective and additively split after inverting the orders of 
stabilizers in the reduction.  In particular, for a free quotient, it is surjective integrally. We  generalize 
this to a splitting of $MU$-module spectra. We also give a stable version of Kirwan's equivariant formality 
theorem: after inverting the torsion primes for $G$, the $G$-action on $MU\wedge M_+$ is trivalizable. 
The novel idea is to exploit the Atiyah-Bott argument in Morava $K$-theory, then return to bordism and 
cohomology.
   \end{abstract}
\maketitle

\section{Introduction} 

Throughout the paper, $G$ will be a compact connected Lie group and $(M,\omega)$ a compact symplectic manifold 
with Hamiltonian $G$-action and moment map $\mu: M \to \mathfrak{g}^*$. 
\subsection{Surjectivity and splitting.}
Recall the now-classical ``Kirwan surjectivity'' theorem: 

\begin{thm*} [\mbox{\cite[\S5]{Kirwan}}]\label{thm:Kirwan} 
The restriction map 
$H^*_G(M; \mathbb{Q}) \to  H_G^*(\mu^{-1}(0); \mathbb{Q})$ is surjective.   
\end{thm*} 
\noindent
The question of integral improvements has surfaced periodically in the literature, e.g.~\cite{tw, aim}.
In this respect, we prove the following sharper result:

\begin{thm} \label{thm:main} 
Suppose that $G$ acts on $\mu^{-1}(0)$ with finite stabilizer groups.  Let $\ell$ be the least common multiple of 
their orders.  Then the map 
\begin{align} 
H^*_G(M; \mathbb{Z}[\ell^{-1}]) \to H_G^*(\mu^{-1}(0); \mathbb{Z}[\ell^{-1}]) \cong 
H^*(\mu^{-1}(0)/G; \mathbb{Z}[\ell^{-1}]) 
\end{align} 
is surjective, and additively split.   
\end{thm} 
\noindent
In particular,  if $G$ acts freely on $\mu^{-1}(0)$,  Kirwan's map is split-surjective on integral cohomology. 
The original arguments in \cite{Kirwan} give a slightly weaker result, with $\ell$ including the orders of 
(components of) additional stabilizers in $M$, not present in $\mu^{-1}(0)$. 

Kirwan deduces Theorem~\ref{thm:Kirwan} from the \emph{equivariant perfection} of the the Morse stratification 
\[
M=\bigcup\nolimits_\beta S_\beta
\] 
associated to the norm-squared function $||\mu||^2$ (with respect to any invariant inner product on~$\mathfrak{g}^*$). 
This stronger statement follows from the \emph{Atiyah-Bott lemma} (\cite[Proposition 13.4]{AB} or 
\cite[Lemma 2.18]{Kirwan}): 
the stratification is equivariantly perfect if the equivariant Euler classes 
$e_G(N_\beta) \in H^*_G(S_\beta; \mathbb{Q})$ of the normal bundles $N_\beta$ to the strata are not zero-divisors. However, these Euler classes can be zero-divisors 
on the torsion part of cohomology, and the argument fails there.  

Instead, we execute Kirwan's argument in certain generalized cohomology theories, the \emph{Morava $K$-theories} 
$K_p(n)$. These are complex-oriented theories which depend on a prime~$p$ and a height~$n$.  We must also 
consider their $r$-fold extensions $K_{p^r}(n)$.  Our equivariant Euler classes are never zero-divisors in 
$K_{p^r}(n)$:   this reduces to the same fact for the Morava Euler classes of non-trivial line 
bundles in $K_{p^r}(n)^*\left(\mathbb{CP}^\infty\right)$.  
Equivariant perfection of the Morse stratification for $K_{p^r}(n)$ follows. For (genuine) 
equivariant complex $K$-theory, this good behavior of Euler classes had already been observed, and used 
to the same effect, in~\cite{Harada-Landweber}. 
Return to ordinary cohomology, in Theorem~\ref{thm:main}, is then effected by noting that the 
Morava $K$-theory of a finite-dimensional $CW$-complex agrees with its periodicized cohomology with 
$\mathbb{Z}/p^r$ coefficients, if the height $n$ is sufficiently large. 

Additional input from homotopy theory entails the perfection of the Morse stratification for 
complex cobordism $MU$ and related complex-oriented theories. Denote by $\landw$ one of 
$K_{p^r}(n)$, or a \emph{Landweber-exact} cohomology theory (\S\ref{sec:mumodules}), and by 
$\landw_G^*(X)$ the same theory applied to the the \emph{Borel construction} $X_G:=X\times_G EG$ 
of a space $X$.\footnote{Our statements thus pertain to the completed, not the genuine equivariant 
theories.} 
\begin{thm}  \label{thm:MoravaKirwan}  
The Morse stratification $M=\bigcup_\beta S_\beta$  is equivariantly perfect for $\landw$.  
In particular,  the  restriction map $\landw^*_G(M)\to  \landw^*_G(\mu^{-1}(0))$ is surjective. 
\end{thm} 
\noindent
We do \emph{not} require here that $G$ should act with finite stabilizers on $\mu^{-1}(0)$. This  
marks a fine distinction between \emph{perfection} of the stratification and the \emph{splitting} 
in Theorem~\ref{thm:main}. Spliting is the shadow of a stronger stable homotopy result; the master 
statement applies to complex bordism. We must invert the 
orders of stabilizers; $MU[\ell^{-1}]$ will denote the localization of $MU$ away form $\ell$. 

\begin{thm} \label{thm:specKirwan}
(Assumptions as in Theorem~\ref{thm:main}.) The map of spectra induced by inclusion is $MU$-linearly split: 
\[
MU [\ell^{-1}]\wedge \mu^{-1}(0)_G \to MU[\ell^{-1}]\wedge M_G.
\] 
\end{thm}
\subsection{Equivariant formality.}
Kirwan also proved \cite[Proposition~5.8]{Kirwan} a rational \emph{equivariant formality}  
for the $G$-action on $M$:
\begin{equation}\label{eq:factorKirwan}
H^*\left(M_G;\mathbb{Q}\right) \cong H^*(M)\otimes H^*(BG;\mathbb{Q}). 
\end{equation}
Using Morava-valued  pseudoholomorphic curve count, this was generalized in \cite{BaiPomerleano} 
to complex-oriented cohomology theories $E$ for which $E^*(BG)$ is free over $E^*$; the method was inspired 
by the similar application in \cite{AMS} to Hamiltonian bundles over $S^2$. Using only  topological 
methods, we prove a spectrum refinement. We must invert the product $\ell_G$ of torsion 
primes for~$BG$. We  indicate the addition of a disjoint base-point to a space by a ``+'', 
placed as typographically convenient. 

\begin{thm}\label{thm:specformal}
The action of~$G$ on $MU[\ell_G^{-1}]\wedge M_+$ is trivializable: there is an $MU$-linear homotopy 
equivalence of bundles of spectra over $BG$
\[
\left(MU[\ell_G^{-1}]\wedge M_+\right)\times_G EG \cong \left(MU [\ell_G^{-1}]\wedge M_+\right) \times BG.
\]
\end{thm}
\noindent
Mapping out from $MU[\ell_G^{-1}]$ to an~$E$ as above and comparing homotopy groups recovers the 
result in~\cite{BaiPomerleano},  $E^*(M_G) \cong E^*(M)\otimes_{E^*} E^*(BG)$.

As expected, we can be more precise when $G$ is a torus $T$. Denote by $(F_\beta)_{\beta \in \mathcal{B}} $  
the connected components of the fixed point locus of the $T$-action and by $\lambda_\beta$ the 
Morse index of~$F_\beta$, when viewed as a critical set for the moment map associated 
to a chosen, but generic circle $S^1 \subset T$.  

\begin{thm} \label{thm:spectorus}   
There is a fiber homotopy equivalence of $MU$-linear spectra over $BT$:
\[
\left(MU\wedge M_+\right)\times_T ET \cong 
\left(\bigvee\nolimits_{\beta \in \mathcal{B}} \Sigma^{\lambda_\beta} MU\wedge F_\beta^+\right) \times BT.
\]
\end{thm}

\subsection{Limitations and counterexamples.}
We indicate here why some restrictions in our theorems are needed, and where improvements could be made. 
\begin{enumerate}\itemsep0ex
\item We do \emph{not} produce preferred splittings in our theorems. Using pseudoholomorphic 
curve methods, the authors of \cite{BaiPomerleano} established the complex-oriented, generalized 
cohomology  version of Theorem~\ref{thm:spectorus}, with preferred splittings for $K(n)$-local 
theories.  It seems likely that a variant of the argument in \cite{BaiPomerleano} using the invariants of \cite{AbouzaidBai},  which produces pseudoholomorphic curve counts directly in $MU$,  would lead to a preferred splitting over $MU$.  Similarly, for actions on \emph{monotone} symplectic manifolds 
and their natural Hamiltonian lifts, a preferred choice of splitting in Theorem~\ref{thm:main} can be 
derived from the Floer construction in \cite{PomTel}. 
\item For ordinary homology, Theorem~\ref{thm:specKirwan} does not truly improve upon Theorem~\ref{thm:main}: 
Eilenberg-MacLane and homology splittings are equivalent.   
\item Theorems~\ref{thm:specformal} and~\ref{thm:spectorus} have alternative formulations in terms 
of global sections of the  bundles of spectra (the homotopy fixed-point spectra), 
regarded as modules over the spectrum of maps $\mathrm{Map}\!\left(BG;MU[\ell_G^{-1}]\right)$; 
see~\cite{MNN}.

\item As in  Kirwan's original theorem,  the results apply more generally to symplectic manifolds 
with proper moment map $\mu: M \to \mathfrak{g}^*$ such that the critical set of $||\mu||^2$ is compact.
\end{enumerate}

The next example shows the need to invert $\ell$ in Theorem~\ref{thm:main} (and~\ref{thm:specKirwan}).

\begin{ex} Consider $\mathbb{C} \times \mathbb{C}$ with $S^1 \times S^1$ acting as a product with weight 2 on each factor. 
Then the stacky quotient is an orbifold and we have an isomorphism  
\[
H^*_T(\mu^{-1}(0); \mathbb{Z}) \cong H^*(B(\mathbb{Z}/2\mathbb{Z})^2;\mathbb{Z}). 
\]
By the K\"unneth formula,  
\[
H^{\operatorname{odd}}(B(\mathbb{Z}/2\mathbb{Z})^2;\mathbb{Z}) \neq 0;
\] 
so there is no Kirwan surjection. 
\end{ex} 

The next two examples show the need for complex-oriented theories in Theorems~\ref{thm:specKirwan} and~\ref {thm:specformal}.

\begin{ex}
Consider the Hamiltonian action of $G=S^1$ act on $M=\mathbb{CP}^2$ induced from the weights 
$(0,1,1)$ on $\mathbb{C}^3$. Reduction at $\mu =1/2$ has quotient $Q=\mathbb{CP}^1$; the unstable 
strata are the fixed-point loci: the point $[1,0,0]$ and the line $[0,*,*]$. The map 
$Q\subset M_G$ is equivalent to the midpoint inclusion in the two-sided mapping cylinder 
from $Q$ to $\mathbb{CP}^\infty$ and to $\mathbb{CP}^\infty\times \mathbb{CP}^1$,
\[
Q\times\{1/2\} \subset \left(Q\times[0,1]\right) 	\:\textstyle{\coprod}_{Q\times\{0\}} \mathbb{CP}^\infty 
	\:\textstyle{\coprod}_{Q\times\{1\}} 
	\mathbb{CP}^\infty\times \mathbb{CP}^1,
\]
for the natural embedding $j$ and for $j\times\mathrm{Id}$, respectively. This inclusion is not stably 
split, because $Sq^2$ on $H^2(\mathbb{CP}^\infty)$ is non-zero, while  vanishing on $Q$.
\end{ex}

\begin{ex}
The natural action of $G=SU_2$ on $M=\mathbb{CP}^1$ has homotopy quotient~$M_G = BS^1$. 
This is not stably equivalent to $\mathbb{CP}^1\times BSU_2$, by the same $Sq^2$ observation. The homotopy quotients of the suspension spectrum $\Sigma^\infty \mathbb{CP}^1$ under 
the natural and trivial actions of $SU_2$ are inequivalent.  
\end{ex}

\begin{ex}
When $BG$ has torsion, the factorization \eqref{eq:factorKirwan} fails for the flag variety $M$, showing 
the need to invert $\ell_G$. 
\end{ex}

\subsection{Acknowledgments} 
D.P.~was partially supported by the NSF grant DMS-2306204.  He would like to thank Neil Strickland for 
correspondence concerning Morava $K$-theory,  Oscar Randall-Williams for clarifications 
concerning Appendix \ref{section:Noetherian},  and Shaoyun Bai for helpful discussions.  
C.T.~was partially supported by the Simons Collaboration of Global Categorical Symmetries, and the 
writing was completed during his stay at the KITP, partially supported by grant NSF-PHY~2309135.

\section{Quick review of $K_{p^r}(n)$, $\En$ and $BP$}

We recall some homotopy-theory background we will use. Some references are \cite{Ravenel}, 
\cite[\S1]{HoveyStrickland};   a gentler introduction to Morava theories is in \cite{Wurgler}.  

\subsection{Formal groups and $MU$.}
The complex cobordism ring spectrum $MU$ has coefficient ring 
\[
MU^*(pt) = \mathbb{Z}[x_1,x_2,\dots], \quad |x_n| = -2n;
\] 
over $\mathbb{Q}$ (but not integrally), we may take the $\mathbb{CP}^n$ as algebra generators. Quillen  
identified~$MU^*$ with Lazard's ground ring of the universal formal group law\footnote{The rings are graded 
and the group coordinate has degree $2$; the group law has degree $0$.} over $\mathbb{Z}$: 
the  group is $\mathrm{Spf}\, MU^*\left(\mathbb{CP}^\infty\right)$ with the natural 
commutative group structure of $\mathbb{CP}^\infty$ and coordinate $u$ given by the hyperplane class. 
The `logarithmic' coordinate change $\log u = \sum_{n\ge 0} [\mathbb{CP}^n]\frac{u^{n+1}}{n+1}$ 
converts the group law into  standard addition, at the cost of including rational coefficients.  

Quillen also showed that the $p$-localization $MU_{(p)}$ of $MU$ at a prime $p$ decomposes as a wedge 
of shifted copies of a complex-oriented, multiplicative cohomology theory $BP$, with coefficient ring
\[
BP^*(pt) = \mathbb{Z}_{(p)}[v_1,v_2,\dots] \quad |v_k| = -2(p^k - 1).
\]
This ring classifies (in a weaker sense) the universal $p$-local formal group law.\footnote
{More precisely: a formal group law over a $p$-local algebra $R$ is canonically isomorphic 
to one derived from a unique homomorphism $BP^*\to R$.} Recall now that a formal group law 
$F$ gives rise to its $l$-series in the ground ring, defined for $l \in \mathbb{Z}$ by: 
\begin{align}
[l]\gdot_{F}(u) =\underbrace{u\gplus_{F} u\gplus_{F} \cdots \gplus_{F} u}_{l \text{ times}}\:;  
\end{align}
\noindent
as written, this only makes sense for $l \geq 0$, but it extends to $l<0$ using the formal group 
inverse.  The generators of $BP^*$ may be chosen to satisfy
\[
[p]\gdot_F (u) = pu \gplus_{F} v_1 u^{p} \gplus_F v_2 u^{p^2} 
	\gplus_F \dots 
\]

\subsection{$MU$-module spectra.}
\label{sec:mumodules}
The spectra  we  consider will be objects in the derived category $D_{MU_{(p)}}$ of module 
spectra over the $p$-localization $MU_{(p)}$ of $MU$ at a prime $p$ \cite[II-III]{EKMM}.  
Thus,~$BP$ is an algebra in $D_{MU_{(p)}}$. Landweber's Exact Functor Theorem 
\cite{Landweber} allows the construction of additional $MU_{(p)}$-module spectra from $BP$: 
those theories, called \emph{Landweber exact}, are defined on spaces by\footnote{Subscript, or homology gradings, are opposite to the superscript cohomology gradings.}
\begin{equation}\label{eq:Landtensor}
\landw_*(X) := BP_*(X)\otimes_{BP^*} \landw_*
\end{equation}
with coefficients in any $BP_*$-module $\landw_*$ in which the sequence $p,v_1,v_2, \dots$ is 
regular. 
The tensoring operation in \eqref{eq:Landtensor} is then exact. Quotienting $BP$ by all 
the~$v_{>n}$ and inverting $v_n$ leads to~$E(n)$,  with coefficients
\[
E(n)^* = \mathbb{Z}_{(p)}[v_1, \dots, v_n, v_n^{-1}], \quad |v_k| = -2(p^k - 1)
\]
and $p$-series
\begin{equation}\label{eq:pseries}
[p]\gdot_F (u) = pu \gplus_F v_1 u^{p} \gplus_F v_2 u^{p^2} \gplus_F \dots \gplus_F v_nu^{p^n}.
\end{equation}

A more general procedure, due to Baas and Sullivan, allows us to quotient out $I_n$ from $E(n)$ 
and define $MU_{(p)}$-module spectra 
\[
K_p(n) = E(n)/I_n = E/I_n.
\]
In particular, we have
\[
K_p(n)^* = E(n)^*/I_n = E^*/I_n = \left(\mathbb{Z}/p \mathbb{Z}\right)[v_n^{\pm 1}].
\]
\noindent
$K_p(n)$ is known as the \emph{Morava $K$-theory} spectrum of height $n$ \cite{johwil}.  
We will also use variants of Morava $K$-theory,  denoted by $K_{p^r}(n)$,  which are $r$-fold  
extensions of the spectrum $K_{p}(n)$ \cite[\S 2]{Strickland}.  The resulting cohomology theories 
are multiplicative,  complex-orientable,   with coefficients 
\[
K_{p^r}(n)^{*} =\left(\mathbb{Z}/p^r \mathbb{Z}\right)[v_n,v_n^{-1}]. 
\]   

Bousfield localization of the spectra $E(n)$ at $K(n)$ defines new ring spectra $\En$, whose coefficient ring 
is described as a completion at $I_n$,
\[
\En{}^*(pt) = (E(n)^*)^\wedge_{I_n} = \mathbb{Z}_p[v_1, v_2, \dots, v_{n-1}, v_n^{\pm 1}]^\wedge_{I_n}.
\]
We refer to $\En$ as the \emph{Morava $E$-theory} spectrum. Unlike the $K(n)$, the  $\En$ admit 
$E_\infty$ structures compatible with that of $MU$ (these are even unique \cite[Theorem 8.3]{BakerRichter}).

\begin{rem} \label{rem:ncKtheory} 
The spectra $E(n)$, $K_{p^r}(n)$ can be constructed directly from $MU_{(p)}$, using the universality 
of the formal group law on $MU^*$. This realizes the $MU$-module structure on our theories. For $p>2$,  the ring spectra 
$K_{p^r}(n)$ are homotopy commutative \cite[Thm.~2.6]{Strickland}, whereas for $p=2$  they are only 
homotopy associative. Nevertheless, the image of the natural map $MU^* \to K_{2^r}(n)^*$ is central. 
In particular,  Chern classes of vector bundles define central elements of $K_{2^r}(n)^*$.  
For~$r=1$,  this can be confirmed from the explicit commutation relation \cite[Theorem~1.5]{Wurgler}: 
$K_{2}(n)^*(BU)$ is concentrated in even degrees,  while the Bockstein operation describing the commutator is odd.  
\end{rem}

\subsection{Some key facts.} \label{sec:facts}
In checking the Atiyah-Bott criterion, we will use some results which we collect here.
The first two are found in \cite[\S5]{HKR}. Let $E$ be one of the cohomology theories $\En$ or 
$K_{p^r}(n)$, and let $\frakm_0$ be the unique 
maximal graded ideal in $E^*(pt)$; for Morava $K_{p}(n)$ we understand this to be the zero ideal.  
\begin{enumerate}\itemsep1ex
\item   \label{lem:Hondagrouplaw}
Suppose $l=p^s l'$ with $p \nmid l'$.  Then 
\[ 
[l]\gdot_{E}u= u^{p^{sn}} x \quad \mod {\frakm_0[[u]]},  
\]
where $x \in E^*(pt)$ is a unit.  
\item  Let $\boldsymbol{\mu}_l\subset S^1$ be the group of  roots of unity of order $l=p^s l'$, 
with $p \nmid l'.$  Then $E^*(B\boldsymbol{\mu}_l)$ is a free module over $E^*(pt)$, with basis 
the pull-backs of $1,\cdots,u^{p^{sn}-1} \subset E^*(BS^1)$.   
\end{enumerate}
Relating the completions $\En$ to localizations and integral theories relies on the following application of 
Landweber's criterion \cite{Landweber2}:
\begin{enumerate}\itemsep1ex
\setcounter{enumi}{2}
\item For a finite $G$-CW complex $X$, the Atiyah-Hirzebruch spectral sequence for $MU^*$ converges:
\[
E_2^{p,q} = H^p_G(X; MU^q) \Rightarrow MU^*(X_G). 
\]
More precisely: for each $p$, all differentials originating on $H^p$ have bounded length, and the 
skeletal filtration on $X_G$ induces a complete Hausdorff topology on $MU^*(X_G)$. 
\end{enumerate}
\noindent
Since the differentials on $MU_{(p)}, BP, BP_p$ and $\En$ are induced from $MU$, the conclusions 
apply to all these theories. They also apply to Morava's $K$-theories, for the more basic reason 
that the Morava $K$-groups are finite for skeletally finite spaces.  

\begin{proof}[Proof of \S\ref{sec:facts}.3]
Landweber's criterion\footnote{The criterion applies to any skeletally finite $CW$ complex, such as $X_G$.} 
requires every rational class in $H^*_G(X;\mathbb{Q})$ to be pulled back under a map from $X_G$ 
to torsion-free space. In our case, the completion theorem \cite{AtiyahSegal} ensures that all classes in a fixed $H^{2k}_G(X;\mathbb{Q})$ are 
homogeneous components of Chern characters of elements in $K_G(X)\otimes\mathbb{Q}$, and are thus pulled 
back under a map from $X_G$ to the torsion-free space $BU$. 
For odd classes, we suspend $X$. 
\end{proof}

\section{Equivariant perfection for exotic cohomologies}
\label{sec:perfect}
We now prove the equivariant perfection of the Kirwan stratification for the Morava $K$-theories 
$K_{p^r}(n)$ and for Landweber-exact cohomology theories, such as $MU, BP$, or $\En$, any of which we 
denote by $\landw$. Note that ordinary cohomology is \emph{not} among them. 

Having chosen an invariant inner product on $\mathfrak{g}$, let $T \subset G$ be a maximal torus and 
fix a positive Weyl chamber $\mathfrak{t}_{+} \subset \mathfrak{g}$.    
The critical sets $C_{\beta}$ of the function $||\mu||^2$ are indexed by a partially ordered finite set $\mathcal{B} \subset \mathfrak{t}_{+}$. Even though $||\mu||^2$  need not be a Morse-Bott function, it defines a $G$-invariant 
stratification of $M$, with the negative gradient flow providing an equivariant retraction of each stratum $S_\beta$  onto the respective critical set $C_\beta$ \cite{Lerman}. 

\begin{proof} [Proof of Theorem~\ref{thm:MoravaKirwan}]
We  run the argument of \cite[Theorem 5.4]{Kirwan} to break the long exact sequences arising from 
successively attaching the strata $S_\beta$. Specifically, if $D_\beta$  
denotes the normal  disk bundle to $S_\beta\subset M$, we must check the \emph{short} exactness of the long Gysin 
sequence 
\begin{equation}\label{eq:split}
\dots \to \landw_G^*(D_\beta,\partial D_\beta) \xrightarrow{\ i_\beta\ }   \landw_G^*(D_\beta) \to  
\landw_G^*(\partial D_\beta) \to \dots.
\end{equation}
The structure of the strata and critical sets established in \cite{Kirwan} gives us $C_\beta = G\times_L C_\beta'$, 
where $L\subset G$ is the Levi subgroup centralizing $\beta$. Splitting \eqref{eq:split} is equivalent to checking exactness in  
\[
0 \to \landw_L^*(D_\beta | C_\beta', \partial D_\beta | C_\beta')  \xrightarrow{\ i_\beta\ }   
	\landw_L^*(D_\beta | C_\beta') \to  \landw_L^*(\partial D_\beta | C_\beta') \to 0,
\]
having restricted $D$ and $\partial D$ to the~$C_\beta'$. The first two spaces can be identified with $\landw^{*-\dim D_\beta}_L(C_\beta')$ 
and $\landw^*_L(C_\beta')$, respectively, whereupon $i_\beta$ becomes multiplication by the $L$-equivariant 
Euler class of $D_\beta$. Splitting then follows from the Atiyah-Bott lemma for $\landw$, applied to the 
subgroup $S^1\subset L$ generated by $\beta$, which we prove next.
\end{proof}

\begin{lem} [Atiyah-Bott Lemma for $\landw$] \label{lem:AtiyahBott1} 
Let $L$ be a compact, connected Lie group and $N$ a finite $CW$ complex with $L$ action.  
Let $V\to N$ be an $L$-equivariant complex vector bundle over $N$.   Assume that $L$ contains 
a central circle subgroup $S^1 \subset L$ which fixes precisely the zero section~$N \subset V$.  Then the Euler class $e_L(V)$ is not a zero divisor in  $\landw^*(N_L)$.  
\end{lem} 

\noindent
We first prove this for the theories $\En$ or~$K_{p^r}(n)$, either of which we call $E$. This case suffices 
for the cohomological surjectivity statement, Theorem~\ref{thm:main}, proved in the next section. 
First, a small simplification. Let $L:= (S^1 \times H)/\boldsymbol{\mu}_l$, with a compact,  
connected Lie group $H$ and a centrally embedded 
$\boldsymbol{\mu}_l\subset S^1 \times H$ 
projecting faithfully to $S^1$, let $p^s$ be the largest $p$-power in $l$, and call $u$ and $\omega$ 
the generating classes for $E^2(BS^1)$ and $E^2(BS^1/\boldsymbol{\mu}_l)$, respectively. 
We use the same notation for their respective pull-backs to $B(S^1\times H)$ and $BL$.

\begin{lem}[Reduction to products] \label{lem: finitecovering} 
Let $N$ be a finite CW complex with an action of $L$.  Then,
\[
E^*(N_{S^1\times H}) = 
	E^*(N_{L})[[u]]\:/\left([l]\cdot_E u = \omega\right).
\]
In particular, $E^*(N_{S^1\times H})$ is a free right module of rank $p^{sn}$ over  $E^*(N_{L})$, 
and the pull-back 
\begin{align*} E^*(N_{L}) \to E^*(N_{S^1\times H}) 
\end{align*} 
is injective. 
\end{lem} 

\begin{proof} 
The relation $[l]\cdot_E u = \omega$ is pulled back from the 
$l$-fold multiplication map $BS^1\to BS^1/\boldsymbol{\mu}_l$. 
To see the isomorphism, note the fibration
\begin{equation} \label{eq:fiberBA} 
B\boldsymbol{\mu}_l \to  N_{S^1 \times H} \to  N_{L};
\end{equation} 
the composition  
\[
B\boldsymbol{\mu}_l \to  N_{S^1 \times H} \to  BS^1,
\] 
corresponds to the faithful 
projection $\boldsymbol{\mu}_l \to S^1.$  In view of \S\ref{sec:facts}.2,  the classes 
$1,\cdots,u^{p^{sn}-1}$ freely generate the $E^*$-cohomology of any fiber 
of \eqref{eq:fiberBA}.  It now follows from the Leray-Hirsch theorem that 
\begin{align} 
E^*(N_{S^1 \times H}) \cong E^*(B\boldsymbol{\mu}_l)\otimes_{E_{*}} E^*(N_{L}),
\end{align}  
completing the proof. 
\end{proof}

\begin{proof}[Proof of Lemma~\ref{lem:AtiyahBott1} for $E^*$]
By splitting the Lie algebra, we can find a complementary connected subgroup $H \subset L$ as in 
Lemma~\ref{lem: finitecovering}, with $L= (S^1\times H)/\boldsymbol{\mu}_l$. Injectivity in that 
lemma reduces us to the group $S^1 \times H$.  Because $S^1$ acts trivially on $N$,  
\begin{align} 
E^*(N_{S^1 \times H})= E^*(N_H)[[u]]. 
\end{align} 
The skeletal filtration on $E^*(N_H)$ is complete, Hausdorff (\S\ref{sec:facts}.3), and pro-nilpotent 
(always). Assume without loss of generality that $N$ is connected. Recalling that $(E^*,\mathfrak{m}_0)$ 
is a graded Noetherian local ring, it follows that $E^*(N_H)$ is also a graded local ring, with 
graded maximal ideal $\mathfrak{m}$ pulled back from 
$\mathfrak{m}_0$ under evaluation at a 
base-point in $N_H$. Moreover, the $\mathfrak{m}$-adic filtration is Hausdorff.  
With $k = \sum p^{nw}$, where the $w$ run through the $p$-exponents of the $S^1$-weights 
on $V$, we obtain from \S\ref{sec:facts}.1  
\begin{align} \label{eq:restrictioneuler} 
e_{S^1\times H}(V)= u^kx + m \in E^*(N_{S^1 \times H}),  
\end{align}  
where $x \in E^*(pt)$ is a unit and $m\in \mathfrak{m}[[u]]$.  

A non-zero element $\nu \in E^*(N_{H})[[u]]$  lies in $\mathfrak{m}^d[[u]] \setminus \mathfrak{m}^{d+1}[[u]]$ 
for some $d$. If $e(V)\cdot \nu = 0$, then
\[
u^kx\cdot \nu = -m\cdot\nu \in \mathfrak{m}^{d+1}[[u]],
\]
leading to the contradiction $\nu\in \mathfrak{m}^{d+1}[[u]]$. 
\end{proof}

\begin{proof}[Proof of Lemma~\ref{lem:AtiyahBott1} for $BP$] 
From \S\ref{sec:facts}.3, $BP(N_L)$ agrees with its limit over the $r$-skeleta $(N_L)_{\le r}$. A 
non-zero $\alpha\in BP^*(N_L)$ is then non-zero on some finite skeleton $X=(N_L)_{\le r}$. 
Generators $v_m$ for large $m$ are prevented from appearing 
in~$\alpha$, by reason of degree. For a large enough height~$n$, we see (from Landweber's 
filtration theorem, or from the explicit estimates in \cite{jw2}) that~$\alpha$ has non-zero 
images~$\alpha(n)\in E(n)^*(N_L)$ and $\widehat{\alpha}(n)\in \En^*(N_L)$. We know that 
$e_L(V)\cdot\widehat{\alpha}(n)\neq 0$ in $\En^*(N_L)$. But this comes from the class 
$e_L(V)\cdot\alpha\in BP^{*+\dim V}(N_L)$, which therefore cannot vanish. 
\end{proof} 

\begin{rem}{\ }\label{effective}
\begin{enumerate}\itemsep0ex
\item A bound on the height $n$ needed in this proof comes from Theorems~1.1 and~2.2 in~\cite{jw2}: having 
fixed the cohomology degree $*$, $\En$ will detect $\alpha$ on  
$(N_L)_{\le q}$ for $n \sim \log_p(q)$.
\item In \cite{Barthel}, the authors construct a splitting of the natural map
\[
BP_p^*(N_L) \longrightarrow \prod\nolimits_{n>0} \En{}^*(N_L).
\]
This immediately implies the Atiyah-Bott Lemma for $BP_p$, since it holds on the right side. 
However, passing to the uncompleted $BP$ requires knowledge of the good skeletal behavior; see 
also our effective version in Appendix~A.
\end{enumerate}
\end{rem}

\begin{proof}[Proof of Lemma~\ref{lem:AtiyahBott1} for $MU$]
The natural formal group law on $MU_{(p)}$ is classified up to isomorphism by the map  $BP\to MU_{(p)}$. 
The Thom isomorphism pulled back thereunder differs by a unit form the natural one on $MU$, so 
the Euler class on the latter also has no zero-divisors. Turning to $MU$, the collapse of the 
Atiyah-Hirzebruch sequence allows us to check the statement by 
restricting to finite skeleta, 
just as for $BP$. On finite spaces, $MU^*$-cohomology 
commutes with $p$-localization, 
and we conclude the absence of zero-divisors for $MU$.      
\end{proof}

Passing to general Landweber-exact theories requires a stable version of Lemma~\ref{lem:AtiyahBott1}. 
Its proof, a bit more technical, is relegated to Appendix~A. 

\begin{prop}[Stable Atiyah-Bott Lemma for $BP$]
\label{prop:kerstability}
For any $q\ge 0$, there is an $r>q$ so that the kernel of 
$e_L(V)\cdot: BP^*\left(N_L\right) \to BP^{*+\dim V}\left((N_L)_{\le r}\right)$ 
vanishes upon restriction to $(N_L)_{\le q}$. 
\end{prop}

\begin{proof}[Proof of Lemma~\ref{lem:AtiyahBott1} for Landweber-exact theories.]
For a finite $CW$ complex $X$, 
\[
\landw^*(X) = BP^*(X)\otimes_{BP^*}\landw^*.
\]
Strong Atiyah-Hirzebruch convergence ~\S\ref{sec:facts}.3 ensures that  
$\landw^*(N_L) = \lim_r  \landw^*\left((N_L)_{\le r}\right)$ and $\lim^1=0$. Proposition~\ref
{prop:kerstability} ensures that the kernels of $e_L(V)\cdot$ on $BP^*\left((N_L)_{\le r}\right)$ 
stabilize to zero. So, then, do the kernels on $\landw^*$, which are
obtained by tensoring with $\landw^*$. Since kernels commute with limits, it follows 
that $e_L(V)\cdot$ is injective on $\landw^*(N_L)$.
\end{proof}

\section{Integral Kirwan surjectivity}
\label{sec:integralcohomologyKir}

\begin{lem} \label{lem:truncation} Let $X$ be a finite CW complex,  $Y$ a skeletally finite CW complex, 
and $f:X \to Y$ a map.  Fix $p,r$ and suppose that for each height $n$,  
the map $f^*: K_{p^r}(n)^*(Y) \to  K_{p^r}(n)^*(X)$ is surjective.  Then, so is the map on cohomology 
\begin{align} \label{eq:surjective} 
f^*: H^*(Y; \mathbb{Z}/p^r \mathbb{Z}) \to H^*(X; \mathbb{Z}/p^r \mathbb{Z}). 
\end{align}
\end{lem} 
\begin{proof} By cellular approximation,   $f$ factors through the $m$-skeleton 
$Y_{\leq m}$ of $Y$ for $m> \operatorname{dim}(X)$. The restricted map  
\begin{align} \label{eq:Moravasurjective} 
f^*: K_{p^{r}}(n)^* (Y_{\leq m}) \to K_{p^{r}}(n)^* (X) 
\end{align} 
is then surjective.  Choose a large height $n$, so that  $|v_n|>m+1$. The Atiyah-Hirzebruch spectral sequences 
\begin{align} 
H^p(Y_{\leq m}; K_{p^{r}}(n)^{q}(pt)) &\Rightarrow K_{p^{r}}(n)^{p+q} (Y_{\leq m}), \\ 
\nonumber H^p(X; K_{p^{r}}(n)^{q}(pt)) &\Rightarrow K_{p^r}(n)^{p+q} (X), 
\end{align} 
collapse for grading reasons.  In degrees $d \in [0,m]$, $E_\infty$ is identified with $H^d(Y_{\leq m}; 
\mathbb{Z}/p^r \mathbb{Z})$ and $H^d(X; \mathbb{Z}/p^r \mathbb{Z})$, respectively, and the induced filtration 
is trivial.  Surjectivity of \eqref{eq:Moravasurjective} then
implies the same for the map 
\begin{align} \label{eq:kequalsh}
H^*(Y_{\leq m}; \mathbb{Z}/p^r \mathbb{Z}) \to H^*(X; \mathbb{Z}/p^r \mathbb{Z}). 
\end{align} 
Finally, the restriction $H^d(Y; \mathbb{Z}/p^r \mathbb{Z}) \to H^d(Y_{\leq m}; \mathbb{Z}/p^r \mathbb{Z})$ is an isomorphism for  $d <m$.  The lemma follows.     
\end{proof}

\begin{proof}[Proof of Theorem \ref{thm:main}.] 
For clarity, we first address the case where $G$ acts freely on $\mu^{-1}(0)$. 
Lemma~\ref{lem:truncation} and Theorem \ref{thm:MoravaKirwan} applied to $X=\mu^{-1}(0)/G$ and $Y=M_G$ imply 
the surjectivity of the map 
\begin{align} \label{eq:kirwanmapZp} 
H^*_G(M; \mathbb{Z}/p^r\mathbb{Z}) \to  
	H^*(\mu^{-1}(0)/G; \mathbb{Z}/p^r\mathbb{Z}) 
\end{align} 
for all primes $p$ and $r \geq 1.$  The universal coefficients theorem now implies that its integral version
\begin{align} \label{eq:kirwanmapZ} 
H^*_G(M; \mathbb{Z}) \to  
	H^*(\mu^{-1}(0)/G; \mathbb{Z}) 
\end{align} 
is surjective, and that the $p^r$-torsion subgroup of $H^*_G(M; \mathbb{Z})$ surjects onto the $p^r$-torsion 
subgroup of $H^*(\mu^{-1}(0)/G; \mathbb{Z})$.  It follows that  \eqref{eq:kirwanmapZ} is a split epimorphism: 
this applies to any surjection $A_0 \to A_1$ of finitely generated abelian groups which 
is surjective on all $p^r$-torsion subgroups.  

Finite stabilizers  in $\mu^{-1}(0)$ require one more topology input. Let $p$ be any prime which does not 
divide $\ell$, and  denote by $L_{(p)}$ the Bousfield localization functor on spaces  with respect to 
$\mathbb{Z}_{(p)}$-homology \cite{Bousfield}. The map $\mu^{-1}(0)_G \to \mu^{-1}(0)/G$ is a 
$\mathbb{Z}_{(p)}$-local equivalence; by the universal property of localization, it determines a map  
\[
\mu^{-1}(0)/G \to L_{(p)}(\mu^{-1}(0)_G),
\] 
which induces isomorphisms in $K_{p^r}(n)$.  The composition $\mu^{-1}(0)/G  \to L_{(p)}(M_G)$ then factors 
through a finite truncation of $L_{(p)}(M_G)$, and the previous argument applies.   \end{proof} 

\section{Equivariant formality} 
\label{section:formality}
We now use  Morava $K$-theory to deduce an integral refinement of 
Kirwan's equivariant formality. While this follows from  $MU$-formality, Theorem~\ref
{thm:specformal}, proved in the next section, we give here a more elementary counting argument. 
As before, $\ell_G$ is the product of the torsion primes for $BG$. 

\begin{thm}\label{thm:cohomology} 
There is a (non-canonical) isomorphism of $H^*(BG; \mathbb{Z}[\ell_G^{-1}])$-modules 
\begin{align*} 
H_G^*(M;\mathbb{Z}[\ell_G^{-1}]) \cong 
H^*(M;\mathbb{Z})\otimes_{\mathbb{Z}} H^*(BG; \mathbb{Z}[\ell_G^{-1}]). 
\end{align*}   
\end{thm}

We start with a torus and Morava theory. As in Theorem~\ref{thm:spectorus}, choose a generic 
one-parameter subgroup $S^1 \subset T$. The associated moment map is a Morse-Bott function 
with critical sets the $T$-fixed points  $F_\beta$, with Morse indices $\lambda_\beta$. 
The normal bundles to the Morse strata meet the conditions of the Atiyah-Bott Lemma~\ref{lem:AtiyahBott1} 
\cite{Atiyah}, so the stratification is equivariantly perfect for $K_{p^r}(n)$. We note first a 
simplified version of Theorem~\ref{thm:spectorus}.

\begin{prop} \label{lem:ABmomentcomponent}
There is a (non-canonical) isomorphism of $K_p(n)^*(BT)$ modules 
\begin{align} \label{eq:freenessofMorava} 
K_p(n)^*(M_T) \cong \bigoplus\nolimits_\beta K_p(n)^{*+\lambda_\beta}(F_\beta)\otimes_{K_p(n)^{*}} K_p(n)^*(BT).  
\end{align} 
\end{prop}
\begin{proof}  
Ordering of the fixed point 
sets by $\mu$-value,  we can describe $K_{p^r}(n)^*(M_T)$ via a sequence of extensions of the form: 
\begin{align} 
\label{eq:shortexactAB} 0 \to K_{p}(n)^{*+\lambda_\beta}((F_\beta)_T)\to 
	K_{p}(n)^*((M_{\leq \beta})_T)\to K_{p}(n)^*((M_{<\beta})_T)\to 0.  
\end{align}  
Each module  $K_{p}(n)^*((F_\beta)_T)$ is free over $K_p(n)^*(BT)$, so  we can inductively split the 
short exact sequences \eqref{eq:shortexactAB}  to obtain the Proposition.   
\end{proof} 

To prove Theorem~\ref{thm:cohomology}, consider first the Atiyah-Hirzebruch-Leray spectral sequence: 
\begin{align} 
\label{eq:degenerationm1} E_2^{s,t}= H^s(BT, K_{p^r}(n)^t(M)) \Rightarrow K_{p^r}(n)^{s+t}(M_T), 
\end{align}
which converges (\S\ref{sec:facts}.3).

\begin{prop} \label{prop:prMoravacircle} 
The  spectral sequence \eqref{eq:degenerationm1} collapses at $E_2$.    
\end{prop}
\noindent
For a finitely generated $K_{p^r}(n)^*(pt)$ module $Q$, let $|Q|$ denote its cardinality (treating 
$K_{p^r}(n)$ as a periodic theory).  As $M$ is built from suspensions of the $F_\beta$, we have the 
`Morse inequalities'
\begin{align} \label{eq:Morseinequalities} 
|K_{p^r}(n)^*(F)| \geq  |K_{p^r}(n)^*(M)| 
\end{align} 
with $F:= \coprod F_\beta$. Calling $\iota: F \hookrightarrow M$ the  inclusion, we have maps 
\begin{align} 
\iota_*: K_{p^r}(n)^*(F_T) \to K_{p^r}(n)^*(M_T),  \quad \iota^*: K_{p^r}(n)^*(M_{T}) \to K_{p^r}(n)^*(F_{T}).  
\end{align} 
The composition $ \iota^* \iota_*$ is multiplication by the Morava Euler class of the normal bundle $\nu$ 
to the fixed point set $F$. This implies that $\iota_*$ is injective, and already suggests that we 
must have equality in \eqref{eq:Morseinequalities}, as we shall confirm in 
the proof.

\begin{proof}[Proof of Proposition~\ref{prop:prMoravacircle}]
Start with the special case $T=S^1$ and let $F_{S^1,m}$ and $M_{S^1,m}$ denote the restrictions to 
$\mathbb{CP}^m$ of the respective Borel fibrations to $BS^1=\mathbb{CP}^\infty$. Then,  
\begin{align} 
\label{eq:Leraycardinalitybound} 
|K_{p^r}(n)^*(M_{S^1,m})| \leq |K_{p^r}(n)^*(M)|^{m+1}. 
\end{align} 
Write $K_{p^r}(n)^*(F_{S^1,m})=\bigoplus_{j=0}^m K_{p^r}(n)^*(F)u^j$,  where 
$u \in K_{p^r}(n)^2(\mathbb{C}P^m)$ is the generator.  From the form \eqref{eq:restrictioneuler} of the 
Euler class, we extract an integer $k \geq 1$, depending on the $p$-exponents in the weights of 
the action of $S^1$ on $\nu$, but \emph{independent of~$m$}, such that 
\begin{align} 
\label{eq:kernelbound} 
\operatorname{ker}(\iota_m^* \iota_{m*}) \bigcap 
	\bigoplus_{j=0}^{m-k} K_{p^r}(n)^*(F)u^j = 0
\end{align}  
for the restricted $\iota_m$ over $\mathbb{CP}^m$; this is because $xu^k$ is a leading term of the Euler 
class in the $\mathfrak{m}$-filtration. The projection of 
$\operatorname{ker}(\iota_m^* \iota_{m*})$ 
onto the complementary space spanned by powers $u^j$ with 
$m-k< j\le m$ is injective. It follows that
\begin{align}
\label{eq:Eulerinjectivitybound} |K_{p^r}(n)^*(F)|^{m+1}-|K_{p^r}(n)^*(F)|^{k+1} \leq |K_{p^r}(n)^*(M_{S^1,m})|.  
\end{align} 
Combining \eqref{eq:Leraycardinalitybound} with~\eqref{eq:Eulerinjectivitybound} for large $m$ 
shows that we must have equality in \eqref{eq:Morseinequalities}. 

Suppose now that we find a non-trivial differential on some page of the 
sequence~\eqref{eq:degenerationm1}.  Multiplying it  by  powers of the hyperplane class in
$H^2(\mathbb{CP}^m)$ would show that the difference  
\[
|K_{p^r}(n)^*(M)|^{m+1}-|K_{p^r}(n)^*(M_{S^1,m})|
\] 
grows unboundedly with $m$,  contradicting the $k$-bound in \eqref{eq:Eulerinjectivitybound}. 
This concludes the proof for $T=S^1$.  

For a general $T$, we proceed by induction on the rank, splitting  $T=S^1 \times H$. 
Choose  finite-dimensional approximations 
$BT_m =  \mathbb{CP}^m\times BH_m\cong \mathbb{CP}^m \times \cdots \times \mathbb{CP}^m$, 
with corresponding $M_{T,m}$. We claim  the $E_2$ degeneration of all Leray sequences  
\begin{align} \label{eq:moravaprdegm} 
E_2^{s,t}= H^s(BT_m, K_{p^r}(n)^t(M)) \Rightarrow K_{p^r}(n)^{s+t}(M_{T,m}).
\end{align} 
Restricting from $BH$ to $BH_m$, our induction hypothesis implies the collapse of \eqref{eq:moravaprdegm} with $H$ replacing $T$.  Since $M_{H,m}$ is a Hamiltonian $S^1$-space, 
we conclude that \eqref{eq:moravaprdegm} also degenerates with $T$.  
From this,  we deduce the surjectivity of the restriction-to-fiber map for all 
$m$, 
\begin{align} \label{eq:degenerationmorava} 
K_{p^r}(n)^*(M_{T,m}) \to K_{p^r}(n)^*(M)
\end{align} 
ensuring the absence of differentials and collapse in~\eqref{eq:degenerationm1}. 
\end{proof}

\begin{proof} [Proof of Theorem~\ref{thm:cohomology}.]
Start with the case when $G$ is a torus $T$. Fix $p, r$, and choose a cut-off $m>\frac{1}{2}\dim M$, 
so that the restriction 
\begin{align}\label{eq:splittingEGn} 
H^d(M_T;\mathbb{Z}/p^r\mathbb{Z}) \to  H^d(M_{T,m};\mathbb{Z}/p^r\mathbb{Z}) 
\end{align} 
is an isomorphism in  degrees $d \leq \dim(M)$ (notation as in the previous proof). 
Choose a height $n$ so that $|v_n|>\dim(M_{T,m})$. From \eqref{eq:degenerationmorava}, 
we conclude (as in the proof of Theorem~\ref{thm:main}) the surjectivity of the map 
\begin{align} \label{eq:degenerationm} 
H^*(M_{T,m};\mathbb{Z}/p^r\mathbb{Z}) \to H^*(M;\mathbb{Z}/p^r\mathbb{Z}), 
\end{align} 
and thus of the restriction $H^*(M_{T}; \mathbb{Z}/p^r\mathbb{Z}) \to H^*(M;\mathbb{Z}/p^r\mathbb{Z})$.  
It follows again that 
\[
H^d(M_{T}; \mathbb{Z}) \to H^d(M;\mathbb{Z}) 
\] 
is onto (and  split).   The theorem for $T$ now follows from the Leray-Hirsch theorem.  

For a general $G$, let $p$ be a prime not in  $\ell_G$ and let $R=\mathbb{Z}/p^r\mathbb{Z}$. 
The Leray spectral sequences $\lbrace E_*^{s,t}(M_T) \rbrace$ and $\lbrace E_*^{s,t}(M_G) \rbrace$ 
computing $H^*(M_T;R)$ and $H^*(M_G;R)$ have $E_2$ pages 
\[
E_2(M_G) = H^*(M;R) \otimes_{R} H^*(BG; R), \quad E_2(M_T) = H^*(M; R) 
\otimes_R H^*(BT, R).
\]
The map $\pi:BT\to BG$ induces a morphism $\pi^*$ of spectral sequences, which is injective on $E_2$ pages, 
since  $G$ has no $p$-torsion. There are thus no differentials on $E_2(M_G)$. Repeating this page by page 
shows that no differentials can occur, and $E^{s,t}_*(M_G)$ degenerates at the second page. 

The restrictions $H^*(M_G;R) \to H^*(M;R)$ are therefore surjective for all such $p,r$. 
Thus, the integral restriction  
\[ 
H^*(M_{G}; \mathbb{Z}[1/\ell_G]) \to H^*(M;\mathbb{Z}[1/\ell_G]) 
\]
is split-surjective, and the theorem follows, again from Leray-Hirsch.  
\end{proof}

\section{Spectral formality} 
\label{sec:specformal}

We start with a quick refresher on (naive) equivariance for spectra.

\subsection{Homotopy $G$-actions on spectra}
A \emph{homotopy $G$-action}  ($hG$-action for short) on a spectrum $F$  is an $A_\infty$-action 
of $G$ by maps of spectra. Essentially by definition, the bar construction converts this structure 
to a local system $\mathcal{F}$ of spectra with fiber $F$ over the classifying space $BG$ (understood as the 
geometric realization of the simplicial $BG$; see \cite{segal} for background). An equivalent 
but more abstract datum is the promotion of $F$ to a left module over the suspension spectrum 
$\Sigma^\infty G_+$, which the multiplication in $G$ turns into an associative ring spectrum 
(the group ring of $G$ over the stable sphere). Of course, strict $G$-actions define  homotopy 
actions. 

The \emph{homotopy fixed-point spectrum} $F^{hG}$ is the space of sections of $\mathcal{F}$ over $BG$. 
For the cohomology theory $X\mapsto \pi_*F(X)$ represented by $F$, 
the homotopy groups of the spectra 
\[
X\mapsto \left(F\wedge X_+\right)^{hG}, \quad 
	F\left(X\right)^{hG}
\]
give the naive equivariant $F$-homologies and cohomologies of $X$. When the $G$-action on $F$ is trivial, 
the latter spectrum agrees with $F(X_G)$. 

\subsection{Extensions.} 
A \emph{fibration sequence} $F' \rightarrowtail F \twoheadrightarrow F''$ of $hG$-spectra is 
a fiber-wise fibration sequence for the associated bundles over $BG$. Its relative delooping over $BG$, 
$\Sigma \mathcal{F}' \rightarrowtail \Sigma \mathcal{F} \twoheadrightarrow \Sigma \mathcal{F}''$, 
defines a fiber-wise action of $\mathcal{F}'' \cong \Omega\Sigma \mathcal{F}''$ on $\Sigma \mathcal{F}'$, 
in particular, a map $\mathcal{F}''\to \Sigma \mathcal{F}'$, after acting on the base-point section of 
$\ \mathcal{F}'$. As usual, its homotopy class determines the $hG$-fibration sequence up to homotopy, 
identifying it with the pull-back from the tautological $\mathcal{F}' \rightarrowtail * \twoheadrightarrow \Sigma \mathcal{F}'$: a homotopy of classifying maps leads to an equivalence of fibration sequences over $BG$. 
In particular, the sequence is split when the classifying map is null-homotopic.
In our case, when the $F$ and the fibration maps are $MU$-modules, the classifying map for the extension
is  $MU$-linear. 

\begin{rem}
Our base $BG$ is a finite type $CW$ complex, and the Atiyah-Hirzebruch sequences stabilize 
(\S\ref{sec:facts}.3); the classification of spectral extensions then follows by the methods 
of fiber-wise homotopy theory over compact bases, as in \cite{crabbjames}: we can check the 
statements cell by cell over the base. More general situations require care in setting up the 
fiber-wise stable homotopy category. 
\end{rem}

\subsection{Proofs.} We now prove Theorems~\ref{thm:specformal} and~\ref{thm:spectorus}; the second one  
will be settled first, leading to the first by a general splitting observation. The argument converts the 
Atiyah-Bott lemma into a vanishing of extension classes.

\begin{proof} [Proof of Theorem~\ref{thm:spectorus}]
As in the \S\ref{section:formality}, a generic circle $S^1\subset T$ will have the same fixed-points 
on $M$ as $T$; denote by 
$\{F_\beta\}_{\beta\in B}$ the connected components of the fixed locus. The moment map $\mu$, 
restricted to~$S^1$, defines a Morse-Bott function with critical loci~$F_\beta$ and Morse 
indices~$\lambda_\beta$. The~$\beta$ are  ordered by their $\mu$-values, with the lowest one 
representing the open Morse stratum. Let $M_{<\beta} \subset M_{\le \beta}$ denote the union 
of Morse strata below, respectively up to $\beta$; these are open in $M$, 
while~$S_\beta: = M_{\le \beta} \setminus M_{<\beta} \subset M_{\le \beta}$ is the stratum 
retracting to $F_\beta$. Its normal bundle~$N_\beta$ in~$M_{\le \beta}$ has an almost complex structure, 
carrying  positive $S^1$-eigenvalues only. 

We show that successively attaching the $S_\beta$ splits over $BT$, after smashing with~$MU$; 
the theorem follows by finite induction. The inclusion  $M_{<\beta} \subset M_{\le \beta}$ 
leads to a fibration sequence over $BT$, 
\begin{align}\label{eqn:BTextension}
MU\wedge M^+_{< \beta} \rightarrowtail MU\wedge M^+_{\le \beta} \:
\twoheadrightarrow \Sigma^{N_\beta} MU\wedge F^+_\beta
\end{align}
This extension over $BT$ is classified by an $MU$-linear map, up to homotopy,
\[
MU\wedge F^+_\beta \to \Sigma^{1-\lambda_\beta} MU\wedge M^+_{< \beta},
\]
(having used the Thom isomorphism). This map represents our extension class 
\[
\varepsilon_\beta\in MU^{1+\dim S_\beta}_{T,cpt}\left(F_\beta \times M_{< \beta}\right),
\] 
having used Poincar\'e duality, with the compactly supported cohomology, on $M_{\beta}$. 

Now, the fibration \eqref{eqn:BTextension} has a natural splitting when pulled back from the Thom space 
of $N_\beta\to S_\beta $ to~$M_{\le \beta}$. Restricting the pull-back extension class to $F_\beta$ 
gives the relation 
\[
e_T(N_\beta)\cdot \varepsilon_\beta = 0 \in MU^{1+\dim M}_{T, cpt}\left(F_{\beta} \times M_{< \beta}\right).
\]
By inductive assumption, $MU\wedge M^+_{< \beta}$ is a sum of shifted constant $MU$-modules over $BT$, and 
then so is its relative $MU$-dual over $BT$ defined by the compactly supported mapping space out of 
$M_{< \beta}$. The Atiyah-Bott Lemma~\ref{lem:AtiyahBott1}
then shows the vanishing of $\varepsilon_\beta$ and completes the induction.
\end{proof}

Passing to a general connected $G$ relies on the following well-known proposition. 
Choose a maximal torus $T\subset G$ and call $\pi:BT\to BG$ the fiber bundle, with fiber $G/T$. 
\begin{prop}
With $\mathbb{Z}[\ell_G^{-1}]$ coefficients, the restriction $H^*(BT) \to H^*(G/T)$ is surjective. 
\end{prop}
\begin{proof}
The Leray sequence for $G/T \rightarrowtail BT \twoheadrightarrow BG$ collapses at $E_2$, being evenly graded. 
\end{proof}
\noindent
The argument applies equally to $MU$. The push-forward $\pi_*MU[\ell_G^{-1}]$ 
of the constant $MU[\ell_G^{-1}]$-bundle over $BT$ is then a direct sum of constant $MU[\ell_G^{-1}]$-bundles over $BG$, 
with basis labeled by a basis of (equivariant extensions of) classes in $MU^*(G/T)$. The canonical map of bundles
\[
MU[\ell_G^{-1}] \to \pi_*MU[\ell_G^{-1}] = \pi_*\pi^* MU[\ell_G^{-1}], 
\] 
is split, up to a unit, by the fiber-wise integration map along $\pi$ against 
the (equivariant extension of the) volume form on $G/T$,
\begin{equation}\label{eq:splitG}
S: \pi_*MU[\ell_G^{-1}] \to MU[\ell_G^{-1}];
\end{equation}
this is because the fiber-wise integration map gives a splitting. This $S$ functions 
algebraically as a section of $\pi$, and will allow us to `restrict' the trivialization 
of the $T$-action to the $G$-action.

\begin{proof}[Proof of Theorem~\ref{thm:specformal}] 
Let $\mathcal{M}$ be the $MU[\ell_G^{-1}]$-linear bundle over $BG$ 
associated to the $G$-action on 
$MU[\ell_G^{-1}] \wedge M$. The pull-back $\pi^*\mathcal{M}$ to 
$BT$ comes from the restricted action. By 
Theorem~\ref{thm:spectorus}, there is an isomorphism 
$MU[\ell_G^{-1}] \wedge M \xrightarrow{\ \sim\ } \pi^*\mathcal{M}$ 
from the constant bundle over $BT$, pushing down to a map 
\[
MU[\ell_G^{-1}] \wedge M \xrightarrow{\ \ } \pi_*\pi^*\mathcal{M} \cong \mathcal{M}\otimes_{MU} \pi_*MU[\ell_G^{-1}].
\] 
Continuing this map with the splitting $S: \pi_*\pi^*\mathcal{M} \to \mathcal{M}$ from \eqref{eq:splitG} gives a 
fiberwise isomorphism $MU[\ell_G^{-1}] \wedge M \xrightarrow{\ \sim\ }\mathcal{M}$, 
trivializing the $G$-action on $MU[\ell_G^{-1}] \wedge M$, as claimed.
\end{proof}

\section{Spectral Kirwan splitting}
We now prove Theorem~\ref{thm:specKirwan}. Again, we use the Atiyah-Bott criterion 
to check the vanishing of successive extension classes upon attaching Morse strata 
$S_\beta$ of $\|\mu\|^2$ to the stable stratum~$S_0$. 

In the notation of \S3, choose an ordering of the critical values so that, for each $\beta$, 
the union 
\[
M_{\le\beta}:= \bigcup\nolimits_{\alpha\le\beta} S_\alpha
\]
is open in $M$; in turn, $S_{\beta}$ is closed in $M_{\le\beta}$. We denote by $\landw$ any 
of  $K_{p^r}(n)$ or Landweber-exact spectrum, such as~$\En, BP$ or $MU_{(p)}$, in which the product $\ell$ 
of orders of stabilizers in $\mu^{-1}(0)$ is invertible. This allows us to treat the quotient orbifold 
$Q$ of $\mu^{-1}(0)$ by $G$ as a compact, almost complex manifold. 

The  inclusion $M_0\subset M_{\le\beta}$ leads to a fibration sequence of $hG$-spectra
\begin{equation}\label{eqn:fibration}
\landw\wedge M_0^+ \rightarrowtail \landw\wedge M_{\le\beta}^+ \twoheadrightarrow 
\landw\wedge \left(M_{\le\beta},M_0\right);
\end{equation}
taking homotopy quotients gives a fibration matching \eqref{eqn:fibration}, 
classified by a map
\begin{equation}\label{eqn:classifying}
\left(\landw\wedge\left(M_{\le\beta},M_0\right)\right)_{hG} \to 
\Sigma \left(\landw\wedge M_0^+\right)_{hG}.
\end{equation}   
The retraction of $M_0$ to $\mu^{-1}(0)$ determines a bundle with fiber $T^*G$ over $Q$; 
this gives an equivalence
\[
\left(\landw\wedge M^+_0\right)_{hG} 
	\cong \landw\wedge Q_+.
\]  
The classifying map in \eqref{eqn:classifying} is $\landw$-linear; using Poincar\'e duality 
on $Q$, it reduces (up to homotopy) to an extension class
\[
\varepsilon_\beta\in \landw^{1-\dim Q}_{G}\left(Q\times(M_{\le\beta},M_0)\right).
\] 
\begin{proof}[Proof of Theorem~\ref{thm:specKirwan}]
The  inclusions $M_0\subset M_{<\beta} \subset M_{\le\beta}$ 
lead to a long exact sequence for the triple 
\[
\dots \longrightarrow  \landw^*_{G}\left(Q\times (M_{\le\beta},M_{<\beta})\right) 
\xrightarrow{\ a\ }  \landw^*_{G}\left(Q\times (M_{\le\beta},M_0)\right)
\xrightarrow{\ b\ }  \landw^*_{G}\left(Q\times (M_{<\beta},M_0)\right) 
\longrightarrow \dots 
\]
The Thom isomorphism identifies the leftmost space with $\landw^{*-\dim N_\beta}_{G}(Q\times S_\beta)$. Now, $b(\varepsilon_\beta) = \varepsilon_{<\beta}$, which we inductively assume to vanish. Then, 
$\varepsilon_\beta = a(\eta)$, with some 
\[
\eta  \in \landw^{*-\dim N_\beta}_{G}\left(Q\times (M_{\le\beta},M_{<\beta})\right) .
\] 
Now, $\varepsilon_\beta$ is null when mapped from the pair $(M_{\le\beta}, M_0)$ 
to $M_{\le\beta}$, as the fibration in~\eqref{eqn:fibration} lifts trivially to the total space. 
Restricting back from $M_{\le\beta}$ to $S_\beta$ leads to 
\[
e_G(N_\beta)\cdot \eta = 0 \in \landw_G^*(Q\times S_\beta), 
\]
forcing $\varepsilon_\beta=0$, for the now familiar reason. Going all the way to the top shows the 
$\landw$-linear splitting of the fibration
\[
\landw\wedge Q_+ \rightarrowtail 
\left(\landw\wedge M_+\right)_{hG} \twoheadrightarrow \left(\landw\wedge (M,M_0)\right)_{hG},
\]
detaching, over $\landw$, the quotient $Q$ from the homotopy quotient $M_G$. 
\end{proof}

\appendix 

\section{Effective Atiyah-Bott lemma for $BP$}

We prove here an explicit version of Proposition~\ref {prop:kerstability}. 
Recall from \S\ref {sec:perfect} 
\begin{itemize}\itemsep0ex
\item $N, V\to N$ with action of 
$L=(S^1\times H)/\boldsymbol{\mu}_l$;
\item the fixed prime $p$;
\item $l=p^sl'$, with $p\nmid l'$;
\item the weights $x_wp^w$ of the central $S^1\subset L$ 
action on~$V$, with $p\nmid x_w$; 
\item the generators $u$ of $BP^2(BS^1)$ and $\omega$ of $BP^2(BS^1/B\boldsymbol{\mu}_l)$.
\end{itemize}
Up to multiplicative units in $BP^*$,
\[
[x_wp^w]\gdot_{BP} u = [p^w]\gdot_{BP} u, \qquad 
	[l] \gdot_{BP} u = [p^s]\gdot_{BP} u.
\]
Writing, in the height $h$ quotient ring $BP^*/I_h[[u]]$, 
\[
[p]\gdot_{BP} u = u^{p^h}\cdot(v_h + v_{h+1}u^{p^{h+1}-p^{h}} + \dots) 
	= u^{p^h}\cdot \pi_h(u),
\]
define $\pi_h[j]: = \pi_h([p^{j}]\gdot_{BP} u)$ (so $\pi_h(u) =\pi_h[0]$), and continue to
\begin{equation}\label{eq:upowers}
\begin{split}
[p^2] \gdot_{BP} u &= \left(u^{p^h}\cdot \pi_h[0]\right)^{p^h} 	\cdot \pi_h\left([p]\gdot_{BP}u\right) = u^{p^{2h}}\cdot 
	\pi_h[0]^{p^h} \cdot \pi_h[0], \\
[p^a] \gdot_{BP} u &= u^{p^{ah}}\cdot \pi_h[0]^{p^{(a-1)h}} \cdot
	\dots \cdot \pi_h[a-1].
\end{split}
\end{equation}

For each skeleton $(N_H)_{< 2q}$, Landweber \cite{Landweber} proves the existence of a $BP^*$-linear filtration 
\[
0=B_0\subset B_1 \subset \dots \subset B_t = 
	BP^*\left((N_H)_{< 2q}\right)
\]
with slices $B_{j+1}/B_j \cong BP^*/I_{n_j}$, for various heights\footnote{From \cite{jw2}, we can bound 
$n_j \le \lceil \log_pq \rceil$; so $k$ is roughtly bounded by $t\sum_w q^w$.} $n_j$.
Let 
\[
B =\sum\nolimits_{w,j} p^{n_jw}.
\]

\begin{prop}
\label{prop:effective}
For each $A$, the kernel of Euler class multiplication on truncated cohomology
\[
\ker e_{S^1\times H}(V)\cdot: BP^*\left((N_H)_{< 2q}
	\times BS^1\right) \to \left.BP^{*+\dim V} 
	\left((N_H)_{< 2q}\times BS^1\right) \right/
	(u)^{A + qB} 
\] 
vanishes mod~$u^A$.
\end{prop}

\begin{proof}
Refining formula~\eqref{eq:restrictioneuler}, we have 
\begin{equation}\label{eq:eulerline}
e_{S^1\times H}(V) = \prod\nolimits_w\left([x_wp^w]\cdot_{BP} u\right) - e_+(u) = \Pi - e_+(u)
\end{equation}
where $e_+(u)\in BP^*(N_H)[[u]]$ vanishes when restricted to~$\left(N_H\right)_{\le 1}\times BS^1$. 
It follows that restriction to $\left(N_H\right)_{< 2q}\times BS^1$ makes $e_+(u)$ nilpotent, $e_+(u)^q =0$. 

The identity
\[
e(V)\cdot \sum\nolimits_{i=1}^q 
	\Pi^{q-i}\cdot e_+(u)^{i-1} = \Pi^q - e_+(u)^q 
= \Pi^q
\] 
then shows that the kernel of reduced multiplication by $e(V)$ 
\begin{equation}\label{eq:kerev}
\ker e(V)\cdot: BP^*\left((N_H)_{< 2q}\right)[[u]] \to 
	BP^*\left((N_H)_{< 2q}\right)[[u]]/(u^M)
\end{equation}
is contained in that of reduced $\Pi^{q}$-multiplication.
Now, a series in 
\[
\ker \Pi^q\cdot: BP/I_h[[u]] \to BP/I_h[[u]]/(u^M).
\]
may not contain $u$-powers below $M-q\sum_w p^{hw}$, as the leading coefficient in $[p^w]$ in~\eqref
{eq:upowers} is not a zero-divisor. It follows by stepping down in the filtration that series 
in~\eqref{eq:kerev} cannot contain $u$-powers lower than $M-qB$. 
This proves the Proposition. 
\end{proof}

To address a general~$L$, consider the $B\boldsymbol{\mu}_l$-fiber bundle $l$-fold multiplication sequence 
\begin{equation}\label{eq:ellfibering}
B\boldsymbol{\mu}_l \rightarrowtail BS^1 
	\xrightarrowdbl{\ [l]\ } B(S^1/\boldsymbol{\mu}_l).
\end{equation} 
Recall that $B\boldsymbol{\mu}_l$ has a $CW$ 
exhaustion by lens spaces $L_r:=S^{2r+1}/\boldsymbol{\mu}_l$;  the action thereon of the circle 
$S^1/ \boldsymbol{\mu}_l$ gives a fiberwise $(2r+1)$-dimensional approximation of the bundle~\eqref
{eq:ellfibering}, homotopy equivalent to the fibration sequence
\[
L_r \to \mathbb{CP}^r \xrightarrowdbl{\ [l]\ } 
	B(S^1/\boldsymbol{\mu}_l). 
\]
Call $\varphi: \hat{X} \twoheadrightarrow X$ the $L_r$-fiber bundle pulled back by a map 
$f:X\to B(S^1 /\boldsymbol{\mu}_l)$, and let $C=\sum_j p^{sn_j}$, with the Landweber heights $n_j$ of $X$. 

\begin{lem}\label{lem:skbound}
If $\dim X<2q$ and $r\ge q(q-1)C$, the pullback in on $BP$-cohomology is injective: 
\[
\varphi^*: BP^*(X) \to BP^*(\hat{X}).
\]
\end{lem}

\begin{proof}
The space $\hat{X}$ can be presented as a circle bundle 
\[
S^1  \rightarrowtail  \hat{X} \twoheadrightarrow 
X\times \mathbb{CP}^r
\]
with $BP$-Chern class $\omega \gplus_{BP} [l]\gdot_{BP}u $ over the base ($\omega$ is pulled back by $f$). Split this as 
$\lambda + \omega $, with $\omega^q=0$ and~$\lambda$ divisible by $u$. The Gysin sequence gives  
\begin{equation}
\dots \to BP^*(X)[[u]]/u^{r+1}	\xrightarrow
	{\ \lambda +\omega\ }
	BP^*(X)[[u]]/u^{r+1} \to  BP^*(\hat{X})\to \dots
\end{equation}
The kernel of $\varphi^*$ consists of classes $\alpha\in BP^*(X)$ expressible as 
\[
\alpha= \beta(u)\cdot(\lambda +\omega)\pmod{u^{r+1}}, 
\]
with some $\beta(u)\in BP^*(X)[[u]]$. We get from the geometric sum formula
\[
\alpha\cdot \sum\nolimits_{i=1}^q 	\lambda^{q-i}\cdot 
	(-\omega)^{i-1} = \beta(u)\cdot\lambda^q  \pmod{u^{r+1}}.
\]
Setting $u=0$ shows that $\alpha\cdot \omega^{q-1}=0$. 
It follows that
\[
\alpha\cdot \sum\nolimits_{i=2}^q 	\lambda^{q-i}\cdot 
	(-\omega)^{i-2} 
	- \beta(u)\cdot\lambda^{q-1}
\]
is in the kernel of $\lambda$ on $BP^*(X)[[u]]/(u^{r+1})$.  The proof of Proposition~\ref{prop:effective} 
then ensures its vanishing mod~$u^{r+1-qC}$. Reducing mod~$u$ again shows the vanishing of 
$\alpha\cdot \omega^{q-2}$; continuing, we conclude that 
$\alpha=0 \pmod{u^{r+1-q(q-1)C}}$, proving that $\alpha=0$.  
\end{proof}

\begin{proof}[Proof of Proposition~\ref{prop:kerstability}]
For any $r$, the kernel of the restricted 
multiplication
\begin{equation} \label{eq:kernelr}
\ker e_L(V)\cdot: BP^*\left(N_L\right) \to 
	BP^{*+\dim V}\left((N_L)_{\le r}\right)
\end{equation}
lifts under $\varphi^*$ to
\begin{equation} \label{eq:kernelrup}
\ker e_{S^1\times H}(V)\cdot: BP^*\left(N_{S^1\times H}\right) \to 
BP^ {*+\dim V}\left((N_{S^1\times H})_{\le r}\right). 
\end{equation}
For any $q'$,  Proposition~\ref {prop:effective} 
finds an $r$ so that \eqref {eq:kernelrup} vanishes on $(N_{S^1\times H})_{\le q'}$. 
Given $q$, Lemma~\ref {lem:skbound} then assures us that the kernel in~\eqref {eq:kernelr}  restricts to zero 
over~$(N_L)_{\le q}$, for suitably large $r$ and $q'$. 
\end{proof}

\section{Noetherian property of $K_{p^r}(n)^*(N_G)$} \label{section:Noetherian}
We prove the following fact, of interest in relation to the proof of 
Lemma~\ref{lem:AtiyahBott1}.

\begin{lem} \label{lem:Noetherian}  For a finite $G$-CW complex $N$,  
the Morava K-theory $K_{p^r}(n)^*(N_G)$ is a Noetherian ring.   \end{lem}
\begin{proof} 
We apply an argument going back to Venkov and Quillen (\cite{Venkov} and \cite[\S 2]{Quillen}).  
Choose a unitary embedding $G \hookrightarrow U(q).$  Composing with the induced map 
$ BG \to BU(q)$ gives rise to a fibration: \begin{align} W \to N_G \to BU(q), \end{align} where 
$W$ has the homotopy type of a finite CW complex.  Consider the Atiyah-Hirzebruch-Leray spectral 
sequence: 
\begin{align} 
E_2^{s,t}= H^s(BU(q), K_{p^r}(n)^t(W)) \Rightarrow K_{p^r}(n)^{s+t}(N_G). 
\end{align} 
Then, $H^*(BU(q),\mathbb{Z}/p^r\mathbb{Z})$ 
is a polynomial algebra in the Chern classes $c_1,\cdots, c_q$. Invoking the finiteness of 
$K_{p^r}(n)^*(W)$ over $K_{p^r}(n)_{*}=\mathbb{Z}/p^r\mathbb{Z}[v_n,v_n^{-1}]$,  we have that the 
$E_2$ page is a finitely generated module over the Noetherian ring $R:=K_{p^r}(n)_{*}[c_1,\cdots,c_q]$.  
It follows from this that the $E_\infty$ page is a finitely generated module over~$R$.  
To conclude,  we note that $K_{p^r}(n)^*(BU(q)) \cong K_{p^r}(n)_{*}[[\hat{c}_1,\cdots,\hat{c}_q]]$,   
where $\hat{c}_1,\cdots,\hat{c}_q$ are the Morava Chern classes.  By construction,  the action 
of $R$ on the $E_\infty$ page is the associated graded of the action of  $K_{p^r}(n)^*(BU(q))$ on 
$K_{p^r}(n)^*(N_G)$.  As the filtration on  $K_{p^r}(n)^*(N_G)$ is exhaustive, Hausdorff, and complete,  
it follows that $K_{p^r}(n)^*(N_G)$ is a finitely generated module over $K_{p^r}(n)^*(BU(q))$
and hence Noetherian.
 \end{proof}

\smallskip
\noindent
{\sc Daniel Pomerleano}, \texttt{daniel.pomerleano@umb.edu} 
\\
{\sc Constantin Teleman}, \texttt{teleman@berkeley.edu}

\end{document}